\documentclass[a4paper,12pt,english]{article}
\usepackage[utf8]{inputenc}   
\usepackage{lmodern}
\usepackage[english]{babel}   
\usepackage[babel]{microtype}
\usepackage{amsfonts,amsmath,latexsym,amsthm,fixmath}
\usepackage{amssymb,ifthen}
\usepackage{authblk}
\usepackage{enumitem}
\usepackage{color,graphicx}
\usepackage{url}
\usepackage[
    pdftex,
    pdfstartview={XYZ 0 1000 1.0},
    bookmarks=false,
    colorlinks,
    breaklinks=true,
    citecolor=purple,
    filecolor=blue,
    linkcolor=blue,
    urlcolor=urlcolor,
    pdfborder={0 0 0}
]{hyperref}
\definecolor{urlcolor}{rgb}{0, 0.5, 0}
\definecolor{citecolor}{rgb}{.5,0,.25}
\definecolor{linkcolor}{rgb}{0,0,1}

\usepackage{geometry}
\usepackage[colorinlistoftodos]{todonotes}

\newtheorem{theorem}{Theorem}

\newtheorem{proposition}[theorem]{Proposition}
\newtheorem{cor}[theorem]{Corollary}
\newtheorem{lemma}[theorem]{Lemma}

\theoremstyle{remark}

\newcommand{\executeiffilenewer}[3]{%
\ifnum\pdfstrcmp{\pdffilemoddate{#1}}%
{\pdffilemoddate{#2}}>0%
{\immediate\write18{#3}}\fi%
}
\def\svgmode{ps}\ifx\pdfoutput\undefined
\def\executeiffilenewer#1#2#3{\immediate\write18{#3}}\else

\ifnum\pdfoutput<1\else\def\svgmode{pdf}\fi\fi

\newcommand{\includesvg}[2][]{%
\def\tempa{#1}\def\tempb{}%
\ifx\tempa\tempb\else\let\svgwidth\tempa\fi
\executeiffilenewer{\svgpath#2.svg}{\svgpath#2.\svgmode}%
{inkscape -z -D --file=\svgpath#2.svg --export-latex %
--export-\svgmode=\svgpath#2.\svgmode}%
\input{\svgpath#2.\svgmode_tex}%
}

\graphicspath{{./figures/}}
\def\svgpath{figures/}

\title{A linear bound for the Colin de Verdi\`ere parameter $\mu$ for graphs embedded on surfaces} 

\author[1]{Camille Lanuel\thanks{This work was conducted when this author was at the G-SCOP laboratory in Grenoble.}} 

\author[2]{Francis Lazarus \thanks{This author is partially supported by LabEx PERSYVAL-Lab (ANR-11-LABX-0025-01) funded by the French program Investissement d’avenir.}}

\author[3]{Rudi Pendavingh}

\affil[1]{Universit\'e de Lorraine, CNRS, Inria, LORIA, F-54000 Nancy, France}
\affil[2]{G-SCOP, CNRS, UGA, Grenoble, France}

\affil[3]{Department of Mathematics and Computer Science, Eindhoven University of Technology, 5600 MB Eindhoven, Netherlands}

\newcommand{\R}{\mathbb{R}}

\definecolor{definecolor}{rgb}{0,0.1,0.55}
\def\define#1{\textbf{\textcolor{definecolor}{#1}}}

\begin{document}

\maketitle

\begin{abstract}
  We provide a combinatorial and self-contained proof that for all graphs $G$ embedded on a surface $S$, the Colin de Verdi\`ere parameter $\mu(G)$ is upper bounded by  $7-2\chi(S)$.
\end{abstract}

\section{Introduction}
\label{sec:Introduction}
In this short note we establish an upper bound for the Colin de Verdière’s graph parameter $\mu$ for graphs that can be embedded on a fixed surface. This parameter was introduced by Colin de Verdière~\cite{c-nigcp-90} in analogy with the multiplicity of the second eigenvalue of Schrödinger operators on a Riemannian surface. The exact definition of $\mu(G)$ resorts to a transversality condition between the space of so-called discrete Schrödinger operators on a graph $G=(V,E)$ and a certain stratification of the space of symmetric matrices of dimension $V\times V$. This \emph{Strong Arnold Hypothesis} (SAH),  as coined by Colin de Verdière~\cite{c-hta-88}, expresses a stability property and ensures that $\mu$ is minor-monotone. It can thus be applied to the graph minor theory of Robertson and Seymour. We refer to the survey by van der Holst, Lovász and Schrijver~\cite{hls-cdvgp-99} for more properties on $\mu$ and the definition of the strong Arnold hypothesis.

Here, we are interested in an upper bound for the parameter $\mu$ of the minor-closed family of graphs that can be embedded on a surface $S$. It is relatively easy to show that $\mu(K_n)=n-1$ for $K_n$ the complete graph
with $n$ vertices~\cite{hls-cdvgp-99}. On the other hand, the largest $n$ such that $K_n$ embeds on $S$ is known as the Heawood number
\[\gamma(S) = \left\lfloor\frac{7+\sqrt{49-24\chi(S)}}{2}\right\rfloor,
\]
where $\chi(S)$ is the Euler characteristic of $S$.
  Colin de Verdière~\cite{c-clpfs-87} conjectured that the maximum of $\mu$ for all graphs that can be embedded in $S$ is attained at $K_{\gamma(S)}$. In other words,  $\mu$ is upper bounded by $\gamma(S)-1$.
  In practice, the known upper bounds have been proved in the realm of Riemannian surfaces where $\mu$ is defined for each Riemannian metric as the maximum multiplicity of the second eigenvalue of Schrödinger differential operators (based on the Laplace-Beltrami operator associated to the given metric). In this framework, Besson~\cite{b-mpvps-80} obtained the bound $7-2\chi(S)$. When the Euler characteristic is negative, this bound was further decreased by 2 by Nadirashvili~\cite{n-melo-88}, who thus showed the upper bound $5-2\chi(S)$. Sévennec~\cite{s-mssec-02} eventually divided by two the dependency in the characteristic to obtain the upper bound $5-\chi(S)$. Note that all those bounds are linear in the genus of $S$ and remain far from the square root bound in the conjecture of Colin de Verdière. It can be proved that any upper bound for $\mu$ in the Riemannian world holds for graphs, cf. \cite[Th. 6.3]{c-sg-98} and \cite[Th. 7.1]{c-clpfs-87}. However, the proof relies on the construction of Schr\"odinger differential operators from combinatorial ones and is not particularly illuminating from the combinatorial viewpoint.
  To our knowledge no purely combinatorial self-contained proof has yet appeared in a peer-reviewed journal.  The goal of this note is to partially fill this gap by proving the bound of Besson in the combinatorial framework of graphs.
\begin{theorem}\label{th:mu-bound}
  Let $G$ be a graph that can be embedded on a surface $S$, then
  \[\mu(G) \leq 7 -2\chi(S).\]
\end{theorem}
Our proof completes and slightly simplifies a proof of Pendavingh that appeared in his PhD thesis~\cite{p-sggc-98}.  Before describing the general strategy of the proof, we provide some relevant definitions and basic facts.

\section{Background}
\paragraph{Schrödinger operators and $\mu$.}
Let $G=(V,E)$ be a connected simplicial\footnote{A graph is \emph{simplicial}, or \emph{simple}, if it has no loops or multiple edges.}  graph with at least two vertices.
A \define{Schrödinger operator} on $G$, sometimes called a generalized Laplacian, is a symmetric $V\times V$ matrix such that for $i\neq j \in V$, its $ij$ coefficient is negative if $ij\in E$ and zero otherwise. (There is no condition on the diagonal coefficients.)
It follows from Perron–Frobenius theorem that the first (smallest) eigenvalue of a Schrödinger operator $L$ has multiplicity one~\cite{c-tccc-94}. Now, if $\lambda_2$ is the second eigenvalue of $L$, then the first eigenvalue of $L-\lambda_2Id$ is negative and its second eigenvalue is zero. This translation by $-\lambda_2Id$ does not change the sequence of multiplicities of the eigenvalues of $L$, nor the stability of $L$ with respect to the strong Arnold hypothesis. Consequently, we can safely restrict to Schrödinger operators whose second eigenvalue is zero. We can now define $\mu(G)$ as the maximal corank (dimension of the kernel) of a Schrödinger operator satisfying the strong Arnold hypothesis. As a fundamental property, $\mu$ is minor-monotone.
\begin{theorem}[\cite{c-nigcp-90}]\label{th:minor-monotone}
   If $H$ is a minor of $G$, then $\mu(H) \leq \mu(G)$.
 \end{theorem}
 It also characterizes planar graphs.
 \begin{theorem}[\cite{c-nigcp-90}]\label{th:mu-planar}
   A graph $G$ is planar if and only if $\mu(G)\leq 3$.
 \end{theorem}
We view a vector of $\R^V$ as a discrete map $V\to \R$, so that a Schrödinger operator acts linearly on the set of discrete maps. For $f: V\to \R$, we denote by $V_f^+, V_f^0, V_f^-$ the subsets of vertices where $f$ takes respectively positive, null and negative values. The support of $f$ is the subset $V_f^+\cup V_f^-$ of vertices with nonzero values. As a simple property of Schrödinger operators we have
\begin{lemma}[\cite{hls-cdvgp-99}]\label{lem:plus-minus}
  Let $L$ be a Schrödinger operator of $G$ and let $f\in \ker L$. Then, a vertex $v\in V_f^0$ is adjacent to a vertex of $V_f^+$ if and only if $v$ is adjacent to a vertex of $V_f^-$.
\end{lemma}
A discrete version of the nodal theorem of Courant reads as follows.
\begin{theorem}[\cite{c-tccc-94,h-sppcc-95}]\label{th:nodal}
  Let $L$ be a Schrödinger operator of $G$ and let $f\in \ker L$ be a nonzero map with minimal support. Then, the subgraphs of $G$ induced respectively by $V_f^+$ and $V_f^-$ are nonempty and connected.
\end{theorem}

\paragraph{Surfaces and Euler characteristic}
By a \define{surface} of finite type we mean a topological space homeomorphic to a compact two dimensional manifold minus a finite number of points. A surface may have nonempty boundary and each of the finitely many boundary components is homeomorphic to a circle. We shall only consider surfaces of finite type and omit to specify this condition. A \define{closed} surface means a compact surface without boundary. By a triangulation of a surface, we mean a simplicial complex together with a homeomorphism between its underlying space and the surface.

Say that a topological space has \define{finite homology} if it has finitely many nontrivial homology groups\footnote{In general, one should consider singular homology for non triangulated spaces.} and each of these are finitely generated. The \define{Euler characteristic} $\chi(X)$ of a space $X$ with finite homology is the alternating sum of its Betti numbers. When $X$ is a finite simplicial complex (and more generally a finite CW complex), this definition coincides with the alternating sum of the numbers of cells of each dimension. The Euler characteristic is homotopy invariant: two spaces with the same homotopy type have the same Euler characteristic.
\begin{proposition}[Inclusion-exclusion formula~{\cite[p.205]{s-at-89}}]\label{prop:inclusion-exclusion}
  Let $Y,Z\subset X$ be spaces with finite homology such that $X = \operatorname{Int}\, Y \cup \operatorname{Int}\, Z$, then
  \[\chi(X)=\chi(Y)+\chi(Z)-\chi(Y\cap Z).
    \]
  \end{proposition}
  Note that this formula might be false if we do not make the assumption that $X$ is the union of the interiors of $Y$ and $Z$. As a counter-example one can take $X$ a line segment, $Y$ a point of $X$ and $Z=X\setminus Y$.
\begin{cor}\label{cor:chi-sum}
  Let $X$ be a triangulated compact surface and let $Y$ be a subcomplex of $X$. Then
  \[\chi(X)=\chi(Y)+\chi(X\setminus Y).
    \]
\end{cor}
\begin{proof}
  Note that $Y$ is closed in $X$ and is a deformation retract of some open subsurface $Z\subset X$. Morever, $Z\cap (X\setminus Y)$ is a union of disjoint annuli or Möbius bands whose Euler characteristic is null. By proposition~\ref{prop:inclusion-exclusion} and the homotopy invariance of $\chi$ we conclude $\chi(X)=\chi(Z)+\chi(X\setminus Y) = \chi(Y)+\chi(X\setminus Y)$.
\end{proof}

\section{Overview of the proof of Theorem~\ref{th:mu-bound}}\label{sec:overview}
Let $G$ be a graph embedded on a surface $S$, which we can assume closed without loss of generality. We may also assume that $S$ is not homeomorphic to a sphere as otherwise Theorem~\ref{th:mu-bound} follows directly from Theorem~\ref{th:mu-planar}. In a first step, we remove an open disk $D\subset S\setminus G$ whose boundary avoids $G$ and build a graph $H$ embedded in $S\setminus D$ such that
\begin{description}
\item[C1.\label{eq:triangulates}] $H$ triangulates $S\setminus D$ and subdivides $\partial D$ into a cycle of $\mu(G)-1$ edges,
\item[C2.\label{eq:minor}] $G$ is a minor of $H$, and
\item[C3.\label{eq:edgewidth}] the length of the shortest closed walk in $H$ that is non-contractible in $S\setminus D$, i.e. the \define{edgewidth} of $H$ in $S\setminus D$, is $\mu(G)-1$. 
\end{description}
We denote by $W$ the set of vertices of $H$. 
We next choose a Schrödinger operator $L$ for $H$ whose corank achieves $\mu(H)$. Condition C2 and the monotonicity of $\mu$ imply $\mu(H)\geq \mu(G)$ so that $\ker L$ has dimension at least $\mu(G)$. It thus contains a nonzero vector $f$ that cancels on the $\mu(G)-1$ vertices of $\partial D$ by C1. We pick such an $f$ with minimal support so that by Theorem~\ref{th:nodal} the subsets of vertices $W_f^+$ and $W_f^-$ induce connected subgraphs of $H$. We connect the vertices of $\partial D$ by inserting $\mu(G)-4$ edges in $D$ to obtain a graph $H'$ with the same vertices as $H$ and that triangulates $S$. We can now extend $f$ linearly on each face of $H'$ to get a piecewise linear map $\bar{f}: S\to \R$. Let $S_f^+, S_f^0, S_f^-$ denote the subspaces of $S$ where 
$\bar{f}$ is respectively positive, null, and negative. By Theorem~\ref{th:nodal}, $S_f^+$ and $S_f^-$ are connected open subsurfaces of $S$, while $S_f^0$ is a closed subcomplex of some subdivision of the triangulation induced by $H'$. We can thus apply Corollary~\ref{cor:chi-sum} to write
\[\chi(S) = \chi(S_f^0) + \chi(S_f^+\cup S_f^-) = \chi(S_f^0) + \chi(S_f^+)+\chi(S_f^-).
\]
By the classification of surfaces, we have $\chi(S_f^+)\leq 1$ and $\chi(S_f^-)\leq 1$. It ensues that $\chi(S_f^0) \geq \chi(S) - 2$. The goal is now to upper bound $\chi(S_f^0)$ in terms of $\mu(G)$ in order to obtain an upper bound for $\mu(G)$. To this end, we build a graph $\Gamma$ with larger Euler characteristic than $S_f^0$ by contracting its two dimensional parts. We then argue that the two dimensional part $K$ containing $D$ has a non-contractible boundary in $S\setminus D$. It follows that this boundary  has length at least the edgewidth of $H$, hence at least $\mu(G)-1$ by condition C3. Thanks to Lemma~\ref{lem:plus-minus} we may infer that $K$ contracts to a vertex of degree at least $\mu(G)-1$ in $\Gamma$. We also argue thanks to Lemma~\ref{lem:plus-minus} that $\Gamma$ has no vertex of degree one. We easily deduce that $\chi(\Gamma) \leq (3-\mu(G))/2$. We finally conclude that
\[(3-\mu(G))/2 \geq \chi(S_f^0) \geq \chi(S) - 2,
\]
hence $\mu(G)\leq 7 -2\chi(S)$. In the remainder of the paper, with provide the details of this sketch of proof.

\section{Triangulation with prescribed edgewidth}
An \define{isometric filling} of the cycle graph of length $n$ is a subdivision $\mathcal F$ of a disk so that $\mathcal F$ has $n$ vertices on its boundary and so that every path in the 1-skeleton of $\mathcal F$ connecting two vertices of $\partial {\mathcal F}$ contains at least as many edges as the shortest path in $\partial {\mathcal F}$ between these two vertices. The isometric filling is triangulated when every face of $\mathcal F$ is a triangle.
\begin{lemma}\label{lem:filling}
  For every $n\geq 3$, the cycle graph of length $n$ has a triangulated isometric filling.
\end{lemma}
\begin{proof}
  First suppose that $n=2m$ is even. We use the following construction of Cossarini~\cite{c-dslam-20}. Consider a simple arrangement ${\mathcal A}$ of $m$ lines in the plane, where any two lines intersect and no three lines have a common intersection\footnote{In fact, any pseudo-line arrangement will do.}. The dual subdivision ${\mathcal A}^*$ has one vertex per face of the arrangement and one edge for each pair of adjacent faces. The union of the bounded faces of ${\mathcal A}^*$ defines a quadrangulation $\mathcal{Q}$ of a disk with $2m$ vertices on its boundary. See Figure~\ref{fig:isofilling}.
  \begin{figure}[h]
    \centering
    \includesvg[.7\linewidth]{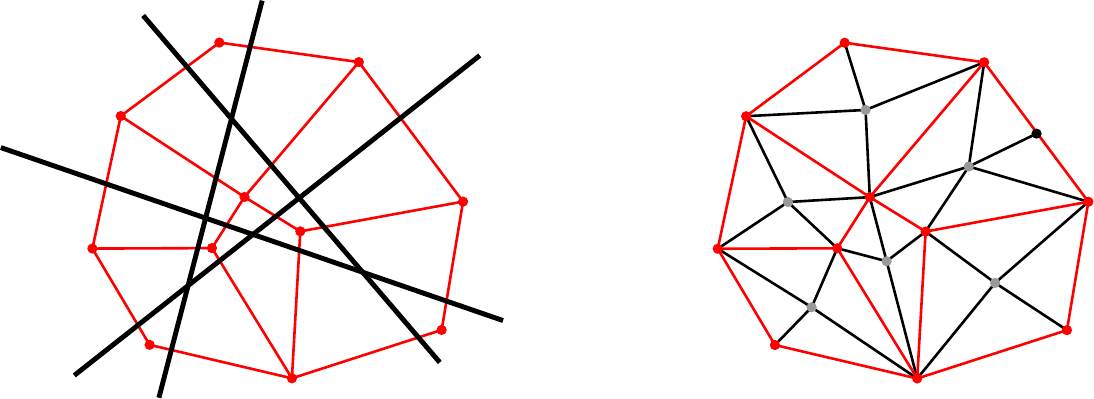}
    \caption{Left, an arrangement of 4 lines and its dual subdivision. Right, an isometric filling of the cycle graph of length 9.}
    \label{fig:isofilling}
  \end{figure}
  This quadrangulation is an isometric filling of the cycle graph of length $n$. Indeed, let $p:u\leadsto v$ be a simple path of length at most $m$ on the boundary of $\mathcal{Q}$. Each edge of $p$ intersects a unique line of ${\mathcal A}$ that separates $u$ from $v$ in the plane. It follows from the Jordan curve theorem that any path $q$ from $u$ to $v$ in $\mathcal{Q}$ also crosses this line. Hence, $q$ contains an edge dual to some edge of ${\mathcal A}$ included in this line. As this is true for every edge of $p$ it ensues  that $p$ is no longer than $q$. In order to get a triangulation, we simply star every quadrangle of $\mathcal{Q}$ from an interior point.

  When $n=2m+1$ is odd, we start with the same quadrangulation $\mathcal{Q}$ as above and insert one vertex on one of its boundary edges. As a result all the faces are quadrangles except for one pentagon. We can nonetheless apply the same triangulation procedure as above by starring each face from an interior point. As this starring defines triangulated isometric filling for a quadrilateral as well as for a pentagon, we obtain this way a triangulated isometric filling of the cycle graph of length $2m+1$.
\end{proof}
Remark that an isometric filling $\mathcal F$ remains so even after  identifying some vertices and edges on the boundary of $\mathcal F$.
Formally, if $\pi: {\mathcal F}\to {\mathcal F}/ \sim$ is the corresponding quotient map, a shortest path in $\pi(\partial {\mathcal F})$ is no longer than any path between the same endpoints in ${\mathcal F}/ \sim$.

\begin{proposition}\label{prop:edgewidth}
  Let $G$ be a graph embedded on a closed surface $S$ that is not a sphere, and let  $D\subset S\setminus G$ be an open disk whose boundary avoids $G$. For every integer $k\geq 3$, there exists a graph $H$, of which $G$ is a minor,  that triangulates $S\setminus D$, subdividing $\partial D$ into a cycle of $k$ edges, and with edgewidth $k$ in $S\setminus D$.
\end{proposition}
\begin{proof}
  We may construct $H$ as follows. We first add a loop edge along the unique boundary component $\partial D$ of $S\setminus D$ and insert other edges to form a cellularly embedded graph $G'$ including $G$ and the loop edge as subgraphs. We next insert $k-1$ vertices in each edge of $G'$ to subdivide it into $k$ subedges. Trivially, the resulting graph $G''$ has edgewidth at least $k$. In fact, it has edgewidth exactly $k$ as $\partial D$ is a non-contractible cycle of length $k$ in $S\setminus D$. It remains to triangulate every face of $G''$ with an isometric filling as in Lemma~\ref{lem:filling}. We can take for $H$ the 1-skeleton of the resulting triangulation of $S\setminus D$. Indeed, by substituting paths interior to the faces of $G''$ by no longer paths on their boundaries we see that every non-contractible closed walk in $H$ is homotopic to a no longer closed walk in $G''$, hence has length at least $k$.
\end{proof}
It is easily seen that the triangulation in the proposition is simplicial.

\section{Proof of Theorem~\ref{th:mu-bound}}
As explained in the proof overview, we may assume that $G$ is embedded on a closed surface $S$ that is not a sphere. We may also assume that $\mu(G)\geq 7$ for otherwise Theorem~\ref{th:mu-bound} is trivially true. We consider an open disk $D$ in $S\setminus G$. By Proposition~\ref{prop:edgewidth}, there is a graph $H$ of edgewidth $\mu(G)-1$, having $G$ as a minor, that triangulates $S\setminus D$ with $\mu(G)-1$ vertices on $\partial D$. Let $L$ be a Schrödinger operator for $H$ with corank $\mu(H)$ and let $f\in \ker L$ of minimal support that cancels on the vertices of $\partial D$. We triangulate $D$ by inserting $\mu(G)-4$ edges with endpoints on $\partial D$. The union of this triangulation with the triangulation of $S\setminus D$ by $H$ defines a triangulation of $S$ whose graph is denoted by $H'$. Note that $H$ and $H'$ have the same set of vertices that we denote $W$. Let $\bar{f}: S\to \R$ be the piecewise linear extension of $f$ and let $S_f^0=\bar{f}^{-1}(0)$. As argued in Section~\ref{sec:overview}, we have
\begin{align}
  \chi(S_f^0) &\geq \chi(S) - 2 \label{eq:chi-S0}
\end{align}
The intersection of $S_f^0$ with a closed triangle of $H'$ may be either the whole triangle, a segment connecting two points on the triangle boundary (either vertices or edge interior points), or one vertex of the triangle. See Figure~\ref{fig:zeroset}.
\begin{figure}[h]
  \centering
  \includegraphics[width=\linewidth]{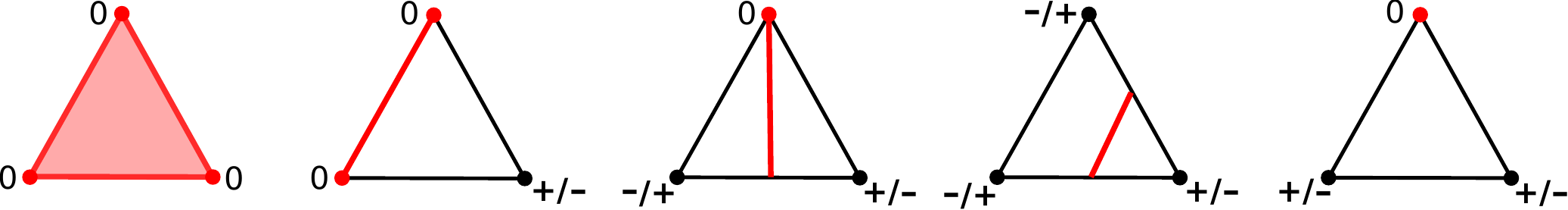}
  \caption{Intersection of the set with a triangle.}
  \label{fig:zeroset}
\end{figure}
$S_f^0$ can thus be decomposed into a two dimensional part made of triangles of $H'$ and a one dimensional part. The two dimensional part is a subsurface of $S$ with singular vertices where several triangles meet at a vertex but do not form a contiguous sequence in the star of the vertex. See Figure~\ref{fig:blowup}, left. We blow up every singular vertex by locally separating the contiguous sequences of triangles and connecting them with a small star graph as on Figure~\ref{fig:blowup}, right.
\begin{figure}[h]
  \centering
  \includegraphics[width=.7\linewidth]{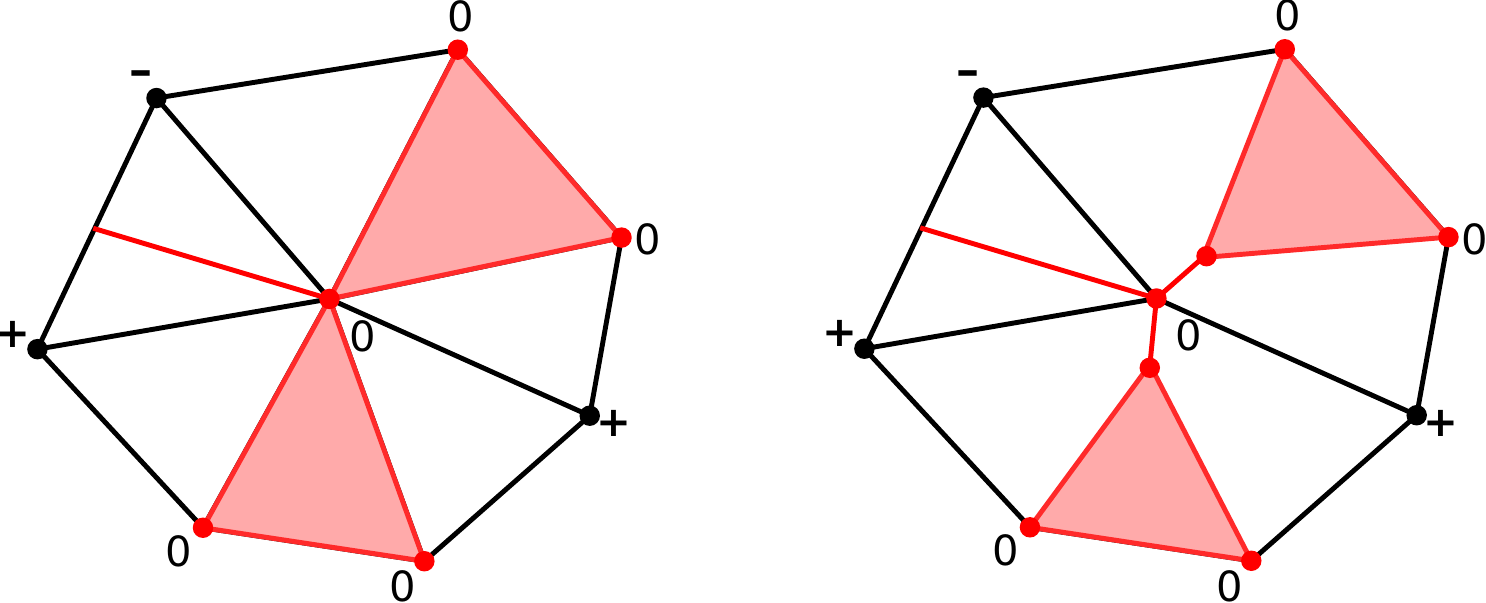}
  \caption{Left, the star of a singular vertex. Right, blowing up the vertex.}
  \label{fig:blowup}
\end{figure}
This does not change the homotopy type of the zero set $S_f^0$ and remove all the singularities of the two dimensional part. Note that this blowup operation neither changes the homotopy (in fact homeomorphism) type of $S_f^+$ and $S_f^-$. We denote by $P, Z, N$ the respective modified sets  $S_f^+, S_f^0, S_f^-$ after blowing up all the singular vertices. Note that $Z$ is a triangulated complex. Now, every two dimensional component of $Z$ is a compact subsurface with nonempty boundary, hence has Euler characteristic at most one.
Moreover, every vertex $v$ on the boundary of such a component is incident to the one dimensional part of $Z$. Indeed, if $v$ results from a blowup operation, then this is true by construction. Otherwise, $v$ is a vertex of $H'$ and Lemma~\ref{lem:plus-minus} implies that $v$ is adjacent to both $W_f^+$ and $W_f^-$ (recall that these are the subsets of vertices of $H'$ where $f$ is respectively positive and negative). As a consequence the link of $v$ intersects $Z$ in at least two components. Since $v$ is not singular at most one of those components is not reduced to a vertex. In particular there must be a vertex component corresponding to a segment of $Z$ incident to $v$.

We next form a graph $\Gamma$ from $Z$ by contracting each two dimensional component $C$ of $Z$ to a vertex $v_C$. By the previous discussion, replacing $C$ by $v_C$, may only increase the Euler characteristic, so that $\chi(\Gamma)\geq \chi(Z)$. Moreover,  the degree of $v_C$ in $\Gamma$ is at least the number of vertices on the boundary of $C$. In particular, this degree is at least three since $H'$ has no loop nor multiple edge (cf. the comment after the proof of Proposition~\ref{prop:edgewidth}). Also, it follows from Lemma~\ref{lem:plus-minus} that $\Gamma$ has no isolated vertices or vertices of degree one. Hence every vertex of $\Gamma$ has degree at least two.

We now consider the two dimensional component $K$ of $Z$ that contains $D$. Let $c$ be a boundary component of $K$. We claim that
$c$ is non-contractible in $S\setminus D$. By way of contradiction, suppose that $c$ is  contractible, hence bounds an open disk $B$ in $S\setminus D$. This disk cannot contain $\partial D$ since it is non-contractible in $S\setminus D$. It follows that $\partial D$, hence $K$, is contained in the complement of $B$ in $S$. On the other hand, by Lemma~\ref{lem:plus-minus}  every vertex of $c$  is in the closure of both $P$ and $N$. Since $P, N$ are both connected by Theorem~\ref{th:nodal}, we deduce that $P$ and $N$ are fully contained in $B$. We are thus in the situation of two disjoint open sets, $P$ and $N$, contained in a disk $B$ and whose closures have at least three vertices $x,y,z$ (from $c$) in common. One can then extract a tripod $K_{1,3}$ in $P$ with leaves $x,y,z$, and similarly for $N$. We connect these two tripods in $B$ by a path $p$ intersecting the tripods only at its extremities. We have thus obtained an embedding in the planar region $B\cup c$ of the union of the two tripods with $p$ and with $c$. This is however impossible as this union graph contains $K_5$ as a minor as illustrated on Figure~\ref{fig:K5}.
\begin{figure}[h]
  \centering
  \includesvg[.6\linewidth]{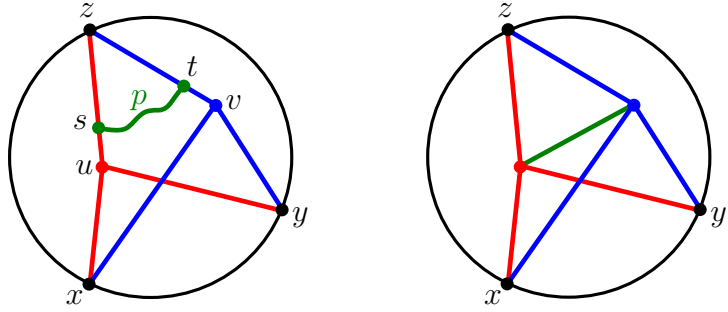}
  \caption{Left, two tripods in a disk and a connecting path. Right, a $K_5$ minor obtained after contracting the edges $su$ and $tv$.}
  \label{fig:K5}
\end{figure}
This ends the proof of the claim. The boundary cycle $c$ of $K$ thus corresponds to a non-contractible closed walk $c'$ in $H$ with the same number of edges as $c$. It ensues that $c$ has length at least $\mu(G)-1$, the edgewidth of $H$. In turn, this implies that $v_K$ has degree at least $\mu(G)-1$ in $\Gamma$. By the handshaking lemma, we have
\[ 2|E(\Gamma)| \geq \mu(G)-1 + 2(|V(\Gamma)|-1).
\]
It follows that 
\[\chi(\Gamma) = |V(\Gamma)| - |E(\Gamma)| \leq |V(\Gamma)| - (|V(\Gamma)| + \frac{\mu(G)-3}{2}).
\]
We infer
\[\chi(S_f^0) = \chi(Z) \leq \chi(\Gamma) \leq \frac{3-\mu(G)}{2}.
\]
In conjunction with the above inequality \eqref{eq:chi-S0}, we conclude $\chi(S)-2 \leq (3-\mu(G))/2$, or equivalently $\mu(G)\leq 7 -2\chi(S)$ as desired.

\bibliographystyle{alpha} 
\bibliography{mu} 
\end{document}